\documentclass[12pt]{amsart}
\usepackage{times}
\usepackage{amsfonts}
\usepackage{amssymb}
\usepackage{bbm}
\usepackage{times}
\usepackage{amssymb}
\usepackage{amscd}
\usepackage{graphicx}

\usepackage{amsmath}
\usepackage{amssymb}

\usepackage{amsmath}
\usepackage{amsfonts}
\usepackage{amscd}
\usepackage{latexsym}
\usepackage{amsthm}
\usepackage{amssymb}
\usepackage{amsmath}
\usepackage{setspace}
\theoremstyle{plain}
\newtheorem{theorem}{Theorem}[section]

\newtheorem{definition}[theorem]{Definition}
\newtheorem{Ques}{Question}

\begin{document}

\title[tracial topological rank zero]{{\bf The inheritance tracial topological rank zero for centrally  large subalgebra}}
\author{Qingzhai Fan}
\address{Qingzhai Fan\\ Department of Mathematics\\  Shanghai Maritime University\\
Shanghai\\China
\\ 201306 }
\email{fanqingzhai@fudan.edu.cn,
qzfan@shmtu.edu.cn}

\author{Xiaochun Fang}
\address{Xiaochun Fang\\ Department of Mathematics\\ Tongji
University\\
Shanghai\\China
\\200092 }
\email{xfang@mail.tongji.edu.cn}


\thanks{2000 \emph{Mathematics Subject Classification\rm{:}} 46L35, 46L05, 46L80}

\begin{abstract}
Let $A$ be an infinite-dimensional unital simple $\rm C^{*}$-algebra such that $A$  has   $\mu$-$oz{\rm LLP}$ property
    	for any $\mu\in(0,1)$. Let $B\subseteq A$ be a centrally large
    	subalgebra of $A$ such that $B$ has  tracial topological rank zero. Then $A$ has  tracial topological rank zero.
\end{abstract}
\maketitle

\section{Introduction}

 Centrally large subalgebra was  introduced in \cite{AN} by  Archey  and  by Phillips as an abstraction of Putnam's orbit breaking subalgebra of the crossed product ${\rm C}^*(X,\mathbb{Z},\sigma)$ of the Cantor set by a minimal homeomorphism in \cite{P}. Whether the properties of centrally  large subalgebra can be preserved in the original algebra has
important applications in studying the properties of  crossed  product  ${\rm C^*}$-algebras which obtained by integer group $\mathbb{Z}$  action
  on compact topological spaces $X$ (\cite{AN}, \cite{AJN}, \cite{ML} and\cite{EN2}).

A crucial step in the classification of nuclear stably finite unital separable ${\rm C^*}$-algebras
 was Lin's  tracial toplogical rank ${\rm C^*}$-algebras (an abstract tracial approximation structure)  which inspired by Elliott-Gong's decomposition theorem  (\cite{E6}) and Gong's decomposition theorem (\cite{G1}) for simple AH algebras.

The following is a question given by Phillips.

\begin{Ques}
Let $A$ be an infinite-dimensional separable unital   simple $\rm C^{*}$-algebra.   Let $B\subseteq A$ be a centrally large subalgebra of $A$ such that $B$ has tracial topological rank zero. Does  $A$ has tracial topological rank zero?
\end{Ques}

 Let $\Omega$ be a class of ${\rm C^*}$-algebras. Inspired  by   centrally  large subalgebra and tracial approximation structure, Elliott, Fan and Fang
 in \cite{EFF} introduced   the class
 of ${\rm C^*}$-algebras which can be weakly  tracially approximated by ${\rm C^*}$-algebras in $\Omega$, and denote this class by ${\rm WTA}\Omega$. This  notion  generalizes  both Archey and Phillips's centrally large subalgebras and Lin's notion of tracial approximation.

 Let ${{\mathbb{F}}}$ denote the class of all finite dimension $\rm C^{*}$-algebras.  In  this paper, we will show the following result.

    	Let $A$ be an infinite-dimensional unital simple $\rm C^{*}$-algebra. Let $B\subseteq A$ be a centrally large	subalgebra of $A$ such that $B$ has tracial topological rank zero. Then  $A\in {\rm WTA}{\mathbb{F}}$.

As an application,  a partially positive answer is  given to Question 1.

Let $A$ be an infinite-dimensional unital simple $\rm C^{*}$-algebra such that $A$  has   $\mu$-$oz{\rm LLP}$ property
    	for any $\mu\in(0,1)$. Let $B\subseteq A$ be a centrally large
    	subalgebra of $A$ such that $B$ has  tracial topological rank zero. Then $A$ has tracial topological rank zero.

    Also one can easily  get the following result.

Let $A$ be an infinite-dimensional separable unital  nuclear simple $\rm C^{*}$-algebra such that $A$ satisfy the $\rm UCT$.   Let $B\subseteq A$ be a centrally large subalgebra of $A$ such that $B$ has tracial topological rank zero. Then $A$ has tracial topological rank zero.

\section{Preliminaries and definitions}
Let $A$ be a ${\rm C^*}$-algebra. For two positive elements $a,b\in A$ we say that $a$ is that $a$ is Cuntz subequivalent to $b$  and write  $a\precsim b$
if there is a sequence $(r_n)_{n=1}^\infty$
of elements of $A$ such that $$\lim_{n\to \infty}\|r_nbr_n^*-a\|=0.$$
 We know that Cuntz equivalent is an equivalence relation. We write $a\sim b$ and say $a$ and $b$ are Cuntz equivalent if $a\precsim b$ and $b\precsim a$.

 We denote by $\mathcal{I}^{(0)}$ the class of finite dimensional ${\rm C^*}$-algebra.

\begin{definition} {\rm (\cite{L2}.)}\label{def:2.2} A  unital simple  ${\rm C^*}$-algebra $A$ is said to
     have tracial topological rank zero if,  for any
 $\varepsilon>0$, any finite
subset $F\subseteq A$, and any  non-zero element $a\geq 0$, there
exist a non-zero projection $p\in A$,   a ${\rm C^*}$-subalgebra $B\in \mathcal{I}^{(0)}$  with
$1_B=p$ such that

$(1)$ $\|px-xp\|<\varepsilon$ for any $x\in F$,

$(1)$  $pxp\in_{\varepsilon} B$ for all $x\in  F$, and

$(2)$ $1-p\precsim a$.

\end{definition}
 If $A$ has tracial topological rank zero, we  write $TR(A)=0$.

Large and centrally  large subalgebras were introduced in \cite{P3} and \cite{AN} by Phillips and  Archey   as  abstractions of Putnam's orbit breaking subalgebra of the crossed product algebra ${\rm C}^*(X,\mathbb{Z},\sigma)$ of the Cantor set by a minimal homeomorphism in \cite{P}.

	\begin{definition}{\rm (\cite{AN}.)} \label{def:2.3}
		Let $A$ be an infinite-dimensional simple unital  $\rm C^{*}$-algebra. A unital subalgebra $B\subseteq A$ is said to be
		centrally large in $A$, if for every $m\in {\mathbb{N}}$,  any $a_{1}, a_{2}, \cdots, a_{m}\in A$, any $\varepsilon>0$, any $x\in A_{+}$,
		with $\|x\|=1$, and any $y\in B_{+}\backslash\{0\}$, there are $c_{1}, c_{2}, \cdots, c_{m}\in A$ and $g\in B$ such that
		
		$(1)$ $0\leq g\leq1$;
		
		$(2)$ For $j=1,2,\cdots,m$, we have $\|c_{j}-a_{j}\|<\varepsilon$;
		
		$(3)$ For $j=1,2,\cdots,m$, we have $(1-g)c_{j}\in B$;
		
		$(4)$ $g\precsim_{B}y$ and $g\precsim_{A}x$;
		
		$(5)$ $\|(1-g)x(1-g)\|>1-\varepsilon$;
		
		$(6)$ For $j=1, 2, \cdots, m$, we have $\|ga_{j}-a_{j}g\|<\varepsilon$.
	\end{definition}


Let $\Omega$ be a class of ${\rm C^*}$-algebras. Elliott, Fan, and Fang in \cite{EFF} introduced   the class
 of ${\rm C^*}$-algebras which can be weakly  tracially approximated by ${\rm C^*}$-algebras in $\Omega$, and denote this class by ${\rm WTA}\Omega$.

    \begin{definition}{\rm (\cite{EFF}.)}\label{def:2.4}
    	A simple unital $\rm C^{*}$-algebra $A$ is said to belong to the class ${\rm WTA}\Omega$ if for any $\varepsilon>0$,
    	any finite subset $F\subseteq A$, and non-zero element $a\geq 0$, there exist a projection $p\in A$, an element
    	$g\in A$, $0\leq g\leq1$, and a unital $\rm C^{*}$-algebra $B$ of $A$ with $g\in B$, $1_{B}=p$ and $B\in \Omega$ such that
    	
    	$(1)$ $(p-g)x\in_{\varepsilon}B$, $x(p-g)\in_{\varepsilon}B$ for all $x\in F$,
    	
    	$(2)$ $\|(p-g)x-x(p-g)\|<\varepsilon$ for all $x\in F$,
    	
    	$(3)$ $1-(p-g)\precsim a$, and
    	
    	$(4)$ $\|(p-g)a(p-g)\|\geq\|a\|-\varepsilon$.
    \end{definition}




The following definition is well-known.
    \begin{definition}\label{def:2.5}
    		Let $A$ be simple exact $\rm C^{*}$-algebra with $T(A)\neq\varnothing$. Then $A$ is said to have
    	strict comparison if for any $a, b\in(A\otimes K)_{+}$, condition
    	\begin{equation*}
    		d_{\tau}(a)<d_{\tau}(b)~~~\text{for all} ~~~\tau\in T(A),
    	\end{equation*}
    	implies that $a\precsim b$.
    \end{definition}

    \begin{definition}{\rm (\cite{NW}.)}\label{def:2.6}
    	Let $A$ be a simple $\rm C^{*}$-algebra with $T(A)\neq\varnothing$. We say $A$ is $a{\rm TAF}$ if,  it satisfies the following
    	condition: For every finite subset $\mathcal{F}\subseteq A$ and $\varepsilon>0$, there exist a finite dimensional
    	$\rm C^{*}$-algebra $F\subseteq A$ and a positive element $a\in F$ with $\|a\|\leq 1$ such that
    	
    	$(1)$ $\|ax-xa\|<\varepsilon$ for all $x\in\mathcal{F}$,
    	
    	$(2)$ $dist(a^{\frac{1}{2}}xa^{\frac{1}{2}},F)<\varepsilon$ for all $x\in\mathcal{F}$, and
    	
    	$(3)$ $\tau(a)>1-\varepsilon$ for all $\tau\in T(A)$.
    \end{definition}

 \begin{definition}{\rm (\cite{NW}.)}\label{def:2.7}
    	Let $A$ be a simple $\rm C^{*}$-algebra with $T(A)\neq\varnothing$. Let $\mu\in (0,1)$. We say $A$ is $\mu$-$a{\rm TAF}$ if it satisfies the following
    	condition: For every finite subset $\mathcal{F}\subseteq A$, there exist a finite dimensional
    	$\rm C^{*}$-algebra $F\subseteq A$ and a positive element $a\in F$ with $\|a\|\leq 1$ such that
    	
    	$(1)$ $\|ax-xa\|<\varepsilon$ for all $x\in\mathcal{F}$,
    	
    	$(2)$ $dist(a^{\frac{1}{2}}xa^{\frac{1}{2}},F)<\varepsilon$ for all $x\in\mathcal{F}$, and
    	
    	$(3)$ $\tau(a)>\mu$ for all $\tau\in T(A)$.
    \end{definition}

 \begin{definition}{\rm (\cite{W}.)}\label{def:2.8}
    	Let $A$ be a simple $\rm C^{*}$-algebra with $T(A)\neq\varnothing$. Let $\mu\in (0,1)$. We say $A$ is $\mu$-${\rm TAF}$, if it satisfies the following
    	condition: For every finite subset $\mathcal{F}\subseteq A$, there exist a finite dimensional
    	$\rm C^{*}$-algebra $F\subseteq A$ and a projection $p\in F$ with $1_F=p$ such that
    	
    	$(1)$ $\|px-xp\|<\varepsilon$ for all $x\in\mathcal{F}$,
    	
    	$(2)$ $dist(pxp,F)<\varepsilon$ for all $x\in\mathcal{F}$, and
    	
    	$(3)$ $\tau(a)>\mu$ for all $\tau\in T(A)$.
    \end{definition}

    \begin{definition}{\rm (\cite{NW}.)}\label{def:2.9}
        A $\rm C^{*}$-algebra $A$ has the $\mu$-$ozLLP$ (for some $\mu\in(0,1)$), if for every finite subset $\mathcal{F}\subseteq A$,
        $\varepsilon>0$, and isometric order zero map $\rho: A\rightarrow \prod_{k}F_{k}/\bigoplus_{k}F_{k}$, where $F_{k}\subseteq A$ are finite dimensional subalgebras and $\tau(x)=\lim_{k\rightarrow\infty}\tau(\rho_{k}(x))$: There exist $M\in{\mathbb{N}}$,
        a complete positive contractive  map $\phi:A\rightarrow \prod_{m>M}F_{m}$, and a positive element of norm one $b\in A$ such that $\phi(b)\neq0$
        and for all $x\in\mathcal{F}$,

        $(1)$ $dist(\phi(bx), \phi(b)'\cap\phi(A))<\varepsilon$,

        $(2)$ $\pi\phi(bx)\approx_{\varepsilon}\rho(bx)\approx_{\varepsilon}\rho(xb)\approx_{\varepsilon}\pi\phi(xb)$, and
        $\pi\phi(b)\approx_{\varepsilon}\rho(b)$, where $\pi:\prod_{m>M}F_{m}\rightarrow\prod_{m>M}F_{m}/\bigoplus_{m>M}F_{m}$
        is the quotient map, and

        $(3)$ $\tau(\phi_{m}(b))>\mu$ for all $\tau\in T(A)$ and $m>M$.

    \end{definition}

    \begin{theorem}{\rm (\cite{L2}.)}\label{thm:2.10}
    	Let $A$ be a separable simple unital  $\rm C^{*}$-algebra with  $TR(A)=0$. Then $A$ has strict comparison.
    \end{theorem}

  \begin{theorem}{\rm (\cite{FFZ1}.)}\label{thm:2.11}
    	Let $A$ be an infinite-dimensional unital simple $\rm C^{*}$-algebra. Let $B\subseteq A$ be a centrally large
    	subalgebra of $A$ such that $B$ has strict comparison. Then
    	$A$ has strict comparison.
        \end{theorem}

The following theorem is well-known.
\begin{theorem}\label{thm:2.12}
    	Let $A$ be a separable simple unital  $\rm C^{*}$-algebra with  $TR(A)=0$. Then $A$ is  tracial $\mathcal{Z}$-stable.
    \end{theorem}

     \begin{theorem}{\rm (\cite{AJN}.)}\label{thm:2.13}
    	Let $A$ be an infinite-dimensional unital simple $\rm C^{*}$-algebra. Let $B\subseteq A$ be a centrally large
    	subalgebra in $A$ such that $B$ is tracial ${{\mathcal{Z}}}$-stable. Then
    	$A$ is $\mathcal{Z}$-stable.
        \end{theorem}

\begin{theorem}{\rm (\cite{HO}.)}\label{thm:2.14}
    	Let $A$ be a separable simple unital nuclear  $\rm C^{*}$-algebra. The the following are equivalent.

     $(1)$ $A$ is $\mathcal{Z}$-stable, and

     $(2)$ $A$  is  tracial $\mathcal{Z}$-stable.

    \end{theorem}

     \begin{theorem}{\rm (\cite{CETWW}.)}\label{thm:2.15}
    	Let $A$ be a separable simple unital $\rm C^{*}$-algebra. The the following are equivalent.

     $(1)$ $A$ is $\mathcal{Z}$-stable, and

     $(2)$ $A$ has finite  nuclear dimension (see \cite{WW3}).

    \end{theorem}

 \begin{theorem}{\rm (\cite{P3}.)}\label{thm:2.16}
    	Let $A$ be an infinite-dimensional unital simple $\rm C^{*}$-algebra. Let $B\subseteq A$ be a centrally large
    	subalgebra in $A$ such that $TR(B)=0$. Then
    	$A$ has stable rank one and real rank zero.
        \end{theorem}

    \begin{theorem}{\rm (\cite{BBSTWW}.)}\label{thm:2.17}
    	Let $A$ be a separable simple unital $\rm C^{*}$-algebra with finite nuclear dimension. Suppose in addition that $A$ satisfy the $\rm UCT$. Then $A$ is stably finite if and only if the decomposition rank of $A$ is finite.
    \end{theorem}

    \begin{theorem}{\rm (\cite{W}.)}\label{thm:2.18}
    	Let $A$ be a separable simple unital  $\rm C^{*}$-algebra such that $A$ has finite decomposition rank and $A$ has real rank zero. Then $TR(A)=0$.
    \end{theorem}

 \begin{theorem}{\rm (\cite{NW}.)}\label{thm:2.19}
    	Let $A$ be a separable unital $\rm C^{*}$-algebra which is $a{\rm TAF}$ and  $A$ has the $\mu$-$oz{\rm LLP}$
    	for some $\mu\in(0,1)$. Then A is $\mu/2$-${\rm TAF}$.
    \end{theorem}

 \begin{theorem}{\rm (\cite{W}.)}\label{thm:2.20}
    	Let $A$ be a separable unital   $\rm C^{*}$-algebra such that $A$ has strict comparsion and   is $\mu$-${\rm TAF}$
    	for any $\mu\in(0,1)$. Then $TR(A)=0$.
    \end{theorem}

 \begin{theorem}{\rm (\cite{NW}.)}\label{thm:2.21}
    	Let $A$ be a separable unital $a{\rm TAF}$ $\rm C^{*}$-algebra such that $A$ has the strict comparison property  and has the $\mu$-$oz{\rm LLP}$
    	for any $\mu\in(0,1)$. Then $TR(A)=0$.
    \end{theorem}

	\section{The main results}

    \begin{theorem}\label{thm:3.1}
    	Let $A$ be an infinite-dimensional  unital simple $\rm C^{*}$-algebra. Let $B\subseteq A$ be a centrally large
    	subalgebra of $A$ such that $B$ has tracial topological rank zero. Let $\Omega$ be a class of finite dimension $\rm C^{*}$-algebra. Then
    	$A\in {\rm WTA}\Omega$.
        \end{theorem}

    \begin{proof}
    	We must show that for any $\varepsilon>0$,
    	any finite subset $F=\{a_{1}, a_{2}, \cdots,$ $a_{m}\}$ $\subseteq A$, and non-zero element $b\geq 0$, there exist a projection $p\in A$, an element
    	$g\in A$, $0\leq g\leq1$, and a unital $\rm C^{*}$-algebra $D$ of $A$ with $g\in D$, $1_{D}=p$ and $D\in \Omega$ such that
    	
    	$(1)$ $(p-g)x\in_{\varepsilon}D$, $x(p-g)\in_{\varepsilon}D$ for all $x\in F$,
    	
    	$(2)$ $\|(p-g)x-x(p-g)\|<\varepsilon$ for all $x\in F$,
    	
    	$(3)$ $1-(p-g)\precsim b$, and
    	
    	$(4)$ $\|(p-g)b(p-g)\|\geq\|b\|-\varepsilon$.
    	
    	Since $A$ is an infinite-dimensional simple unital $\rm C^{*}$-algebra, there exist positive elements
    	$b_{1}, b_{2}\in A_{+}$ of norm one such that $b_{1}b_{2}=0, b_{1}\sim b_{2}$ and $b_{1}+b_{2}\precsim b$.
    	
    	Given $<0\delta<\varepsilon$, since $B$ is a centrally large  subalgebra of $A$, then  for every $m\in {\mathbb{N}}$, $a_{1}, a_{2}, \cdots, a_{m}\in A$, $\varepsilon>0$, $b_{1}, b_{2}\in A_{+}$,
    	 $y\in B_{+}\backslash\{0\}$, there are $c_{1}, c_{2}, \cdots, c_{m}\in A$ and $g'\in B$ such that
    	
    	$(1')$ $0\leq g'\leq1$,
    	
    	$(2')$  $\|c_{j}-a_{j}\|<\delta$ for $j=1, 2, \cdots, m$,
    	
    	$(3')$ $(1-g')c_{j}\in B$ for $j=1, 2, \cdots, m$,
    	
    	$(4')$ $g\precsim_{B}y$ and $g\precsim_{A}b_{1}\sim b_{2}$,
    	
    	$(5')$ $\|(1-g')b_{1}(1-g')\|>1-\delta; \|(1-g')b_{2}(1-g')\|>1-\delta$, and
    	
    	$(6')$ $\|g'a_{j}-a_{j}g'\|<\delta$ for $j=1, 2, \cdots, m$.
    	
    	Let $G=\{(1-g')c_{1}, (1-g')c_{2}, \cdots, (1-g')c_{m}, g'\}\subseteq B$. Since
    	$TR(B)=0$, for $\delta$, $G$ and $b_{2}\in A_{+}$, there exist a non-zero projection $p\in A$ and a $\rm C^{*}$-subalgebra
    	$D\in \Omega$ with $1_{D}=p$ such that
    	
    	$(1'')$ $\|p(1-g')c_{j}-p(1-g')c_{j}p\|<\delta$ for $j=1, 2, \cdots, m$ and $\|pg'-g'p\|<\delta$,
    	
    	$(2'')$ $p(1-g')c_{j}p\in_{\delta}D$ for $j=1, 2, \cdots, m$ and $pg'p\in_{\delta}D$,
    	
    	$(3'')$ $1-p\precsim b_{2}$, and

   	$(4'')$ $\|pg'p\|>1-\delta$.
    	
    	By $(2'')$, there exist $[(1-g')c_{j}]'\in D$ for $j=1, 2, \cdots, m$,
    such that $$\|[(1-g')c_{j}]'-p(1-g')c_{j}p\|<\delta,$$ and exists an element $g\in D$ such that
     $$\|g-pg'p\|<\delta.$$

    	We may assume that  $0\leq g\leq1$.
    	
    	By $(2')$ and $(6')$, we can get
    	
    	$\|(1-g')c_{j}-c_{j}(1-g')\|=\|g'c_{j}-c_{j}g'\|$
    	
    	$\leq\|g'c_{j}-g'a_{j}\|+\|g'a_{j}-a_{j}g'\|+\|a_{j}g'-c_{j}g'\|$
    	
    	$<3\delta<3\varepsilon$.
    	
    	By $(1'')$, we can get $\|(1-g')p-p(1-g')\|<\delta<\varepsilon$.
    	
     $(1)$    	
    	$\|(p-g)a_{j}-[(1-g')c_{j}]'\|$
    	
    	 $\leq\|(p-g)a_{j}-(p-pg'p)a_{j}\|$

    $+\|(p-pg'p)a_{j}-(p-pg'p)c_{j}\|$

    $+\|(p-pg'p)c_{j}-p(1-g')c_{j}p\|$
    	
    	$+\|[(1-g')c_{j}]'-p(1-g')c_{j}p\|$
    	
    	$<\varepsilon+\delta+\delta+\delta<4\varepsilon$ for $j=1, 2, \cdots, m$.
    	
    	Similarly, $a_{j}(p-g)\in_{4\varepsilon}D$ for $j=1, 2, \cdots, m$.
    	
    	 $(2)$    	
    	$\|(p-g)a_{j}-a_{j}(p-g)\|$
    	
    	 $\leq\|(p-g)a_{j}-(p-pg'p)a_{j}\|$

    $+\|p(1-g')a_{j}-p(1-g')c_{j}\|$

    $+\|p(1-g')c_{j}-(1-g')c_{j}p\|$
    	
    	$+\|(1-g')c_{j}p-c_{j}(1-g')p\|+\|c_{j}(1-g')p-a_{j}(1-g')p\|$

    $+\|a_{j}(1-g')p-a_{j}(p-g)\|$
    	
    	$<8\varepsilon$ for $j=1, 2, \cdots, m$.
    	
    	 $(3)$  $1-(p-g)=g+1-p\sim g+(1-p-\varepsilon)_+$

    $\precsim pg'p\oplus (1-p)\precsim g'\oplus (1-p)$

    $\precsim
    b_{1}+b_{2}\precsim b$.
    	
    	 $(4)$, $\|(p-g)b(p-g)\|\geq\|p(1-g')b(1-g')p\|-\delta\geq\|pbp\|-3\delta\geq\|b\|-3\varepsilon$.
    	
    	Therefore, by $(1)$, $(2)$, $(3)$, and $(4)$  one has $A\in {\rm WTA}\Omega$.
    \end{proof}

\begin{theorem}
    	Let $A$ be an infinite-dimensional separable unital simple $\rm C^{*}$-algebra such that $A$  has the $\mu$-$oz{\rm LLP}$ property
    	for any $\mu\in(0,1)$. Let $B\subseteq A$ be a centrally large
    	subalgebra of $A$ such that $TR(B)=0$. Then $TR(A)=0$.
         \end{theorem}
\begin{proof} By Theorem \ref{thm:2.10}, Theorem \ref{thm:2.11},  Theorem \ref{thm:3.1}, Theorem \ref{thm:2.19} and Theorem \ref{thm:2.20}.
\end{proof}

\begin{theorem}
    	Let $A$ be an infinite-dimensional separable unital  nuclear simple $\rm C^{*}$-algebra such that $A$ satisfy the $\rm UCT$.   Let $B\subseteq A$ be a centrally large subalgebra of $A$ such that $TR(B)=0$. Then $TR(A)=0$.
        \end{theorem}
\begin{proof} By Theorem \ref{thm:2.12}, Theorem \ref{thm:2.13}, Theorem \ref{thm:2.14}, Theorem \ref{thm:2.15}, Theorem \ref{thm:2.16}, Theorem \ref{thm:2.17}, and Theorem \ref{thm:2.18}.
\end{proof}

We hope the following question is right.
\begin{Ques} Let $\Omega$ be  the class of finite dimensional  $\rm C^*$-algebras. Let $A\in {\rm WTA}\Omega$ be an infinite-dimensional separable unital simple $\rm C^*$-algebra such that $A$ has  stable rank one, real rank zero and has strict comparison. Does  $A$ has tracial topological rank zero?

\end{Ques}

	\end{document}